\documentclass[oneside,english]{amsart}
\usepackage[T1]{fontenc}
\usepackage[latin9]{inputenc}
\usepackage[a4paper]{geometry}
\usepackage{amsthm}
\usepackage{amstext}
\usepackage{amssymb}
\usepackage{esint}

\makeatletter
\numberwithin{equation}{section}
\numberwithin{figure}{section}
  \theoremstyle{plain}
  \newtheorem*{thm*}{Theorem}
\theoremstyle{plain}
\newtheorem{thm}{Theorem}
  \theoremstyle{plain}
  \newtheorem{cor}[thm]{Corollary}
  \theoremstyle{plain}
  \newtheorem{prop}[thm]{Proposition}
  \theoremstyle{plain}
  \newtheorem{lem}[thm]{Lemma}

\makeatother

\usepackage{babel}

\begin{document}

\title{Probabilistic representation of Bernoulli, Euler, and Carlitz Hermite
polynomials}

\author{C. Vignat, Université d'Orsay}
\begin{abstract}
We revisit in a probabilistic framework the umbral approach of Bernoulli
numbers, Euler numbers and Carlitz Hermite polynomials by Gessel \cite{Gessel}.
This study allows to explicit equivalents of some famous umbr\ae. 
\end{abstract}
\maketitle

\section{Introduction}

In \cite{Gessel}, Gessel shows how the umbral calculus, introduced
by Blissard and later developed by Rota and Roman \cite{Roman}, allows
to derive some elementary but also more elaborate results about Bernoulli
numbers and about classical polynomials such as Charlier and Hermite
polynomials. In this note, we show that another approach based on
probabilistic representations can be substituted to the umbral formalism.
This substitution brings interesting results about the relationship
between random variables and umbr\ae.

\section{Bernoulli numbers and Bernoulli polynomials}

The sequence of Bernoulli numbers $\left\{ B_{n}\right\} $ is defined
by its generating function\[
\sum_{n=0}^{+\infty}B_{n}\frac{t^{n}}{n!}=\frac{t}{e^{t}-1},\,\,\vert t\vert<2\pi\]
and the sequence of Bernoulli polynomials $\left\{ B_{n}\left(x\right)\right\} $
by their generating function\[
\sum_{n=0}^{+\infty}B_{n}\left(x\right)\frac{t^{n}}{n!}=\frac{te^{tx}}{e^{t}-1},\,\,\vert t\vert<2\pi,\, x\in\mathbb{R},\]
so that the Bernoulli numbers verify $B_{n}=B_{n}\left(0\right).$
First values are\[
B_{0}\left(x\right)=1;\,\, B_{1}\left(x\right)=x-\frac{1}{2};\,\, B_{2}\left(x\right)=x^{2}-x+\frac{1}{6};\,\, B_{3}\left(x\right)=x^{3}-\frac{3}{2}x^{2}+\frac{x}{2}\]
and \[
B_{0}=1;\,\, B_{1}=-\frac{1}{2};\,\, B_{2}=\frac{1}{6};\,\, B_{3}=0;\,\, B_{4}=-\frac{1}{30}.\]

In \cite{Sun}, Sun gives the following probabilistic representation
of the Bernoulli polynomials
\begin{thm*}
{[}Sun{]} Given a sequence $\left\{ L_{n}\right\} $ of independent
random variables, each with Laplace distribution $\frac{1}{2}\exp\left(-\vert x\vert\right),\,\, x\in\mathbb{R},$
define the random variable \begin{equation}
L=\sum_{k=1}^{+\infty}\frac{L_{k}}{2\pi k}.\label{eq:L}\end{equation}
Then the following probabilistic representations hold %
\footnote{In this paper, the symbol $E$ denotes the expectation operator $Eg\left(X\right)=\int f\left(X\right)g\left(X\right)dX$
where $f$ is the probability density of the relevant random variable.%
}\begin{equation}
B_{n}\left(x\right)=E\left(\imath L+x-\frac{1}{2}\right)^{n},\,\, n\ge0,\,\, x\in\mathbb{R}\label{eq:Bn(x)Sun}\end{equation}
and \begin{eqnarray}
B_{n} & = & E\left(\imath L-\frac{1}{2}\right)^{n},\,\, n\ge0.\label{eq:Bn}\end{eqnarray}

\end{thm*}
The random variable $L$, being defined as a series of i.i.d. random
variables, is not easy to characterize. Thus we propose the equivalent
result.
\begin{thm*}
The random variable in \eqref{eq:L} follows a logistic distribution,
with density\[
f_{L}\left(x\right)=\frac{\pi}{2}\text{sech}^{2}\left(\pi x\right),\,\, x\in\mathbb{R}.\]
\end{thm*}
\begin{proof}
The random variable $L$ in \eqref{eq:L} has characteristic function\[
E\left(e^{\imath tL}\right)=\frac{\frac{t}{2}}{\sinh\left(\frac{t}{2}\right)}.\]
But from \cite[1.9.2]{Bateman}\[
\int_{0}^{+\infty}\text{sech}^{2}\left(ax\right)\cos\left(xt\right)dx=\frac{\pi t}{2a^{2}}\text{csch}\left(\frac{\pi t}{2a}\right)\]
so that, with $a=\pi,$ the density of $L$ is \[
f_{L}\left(x\right)=\frac{\pi}{2}\text{sech}^{2}\left(\pi x\right).\]

\end{proof}
We note from \cite[p. 471]{Devroye} that the random variable $L$
can be obtained as \[
L=\frac{1}{2\pi}\log\frac{U}{1-U}=\frac{1}{2\pi}\log\frac{E_{1}}{E_{2}}\]
where U is uniformly distributed on $\left[-1,+1\right]$, $E_{1}$
and $E_{2}$ are independent with exponential distribution $f_{E}\left(x\right)=\exp\left(-x\right),\,\, x\in\left[0,+\infty\right[$
and equality is in the sense of distributions.
\begin{cor}
The Bernoulli polynomials read \begin{equation}
B_{n}\left(x\right)=E\left(\frac{\imath}{2\pi}\log\frac{U}{1-U}+x-\frac{1}{2}\right)^{n}=E\left(\frac{\imath}{2\pi}\log\frac{E_{1}}{E_{2}}+x-\frac{1}{2}\right)^{n}\label{eq:BnU}\end{equation}
and the Bernoulli numbers\begin{eqnarray*}
B_{n} & = & E\left(\frac{\imath}{2\pi}\log\frac{U}{1-U}-\frac{1}{2}\right)^{n}=E\left(\frac{\imath}{2\pi}\log\frac{E_{1}}{E_{2}}-\frac{1}{2}\right)^{n},\,\, n\ge0.\end{eqnarray*}

\end{cor}
We remark that the result by Sun was already given by Talacko in \cite{Talacko}. 

We note that $L$ is also \cite{Stefanski} a Gaussian scale mixture\[
L=2K.N\]
where $N$ is Gaussian and $K$ follows the Kolmogorov-Smirnov distribution\[
f_{K}\left(x\right)=8x\sum_{n=1}^{+\infty}\left(-1\right)^{n+1}n^{2}\exp\left(-2n^{2}x^{2}\right),\,\, x\ge0.\]
This result was given first by Barndorff-Nielsen et al. in \cite{Barndorff},
where the mixing distribution of $2K$ is also characterized by its
moment generating function \[
\varphi_{K}\left(s\right)=E^{Ks}=\frac{\pi\sqrt{2s}}{\sin\left(\pi\sqrt{2s}\right)}.\]

\begin{prop}
\label{pro:proposition1}The Bernoulli numbers satisfy\[
B_{n}=E\left(\imath L+\frac{1}{2}\right)^{n}\,\,,\,\, n\ne1.\]
\end{prop}
\begin{proof}
We first compute the generating function\begin{eqnarray*}
\sum_{n=0}^{+\infty}B_{n}\frac{t^{n}}{n!} & = & E\sum_{n=0}^{+\infty}\left(\imath L-\frac{1}{2}\right)^{n}\frac{t^{n}}{n!}\\
 & = & \exp\left(-\frac{t}{2}\right)E\exp\left(\imath tL\right)=e^{-\frac{t}{2}}\frac{t}{e^{\frac{t}{2}}-e^{-\frac{t}{2}}}=\frac{t}{e^{t}-1}\end{eqnarray*}
as it should be. Now let us consider\[
E\sum_{n=0}^{+\infty}\left(\imath L+\frac{1}{2}\right)^{n}\frac{t^{n}}{n!}=\frac{te^{t}}{e^{t}-1}=\frac{t}{e^{t}-1}+t\]
so that, since E$\left(\imath L+\frac{1}{2}\right)=\frac{1}{2},$\[
E\sum_{n=2}^{+\infty}\left(\imath L-\frac{1}{2}\right)^{n}\frac{t^{n}}{n!}=E\sum_{n=2}^{+\infty}\left(\imath L+\frac{1}{2}\right)^{n}\frac{t^{n}}{n!}\]
which shows the result.
\end{proof}
Let us show now that the representation \eqref{eq:Bn} allows to recover
easily some fundamental results about Bernoulli numbers.

In umbral calculus, as described in \cite{Gessel}, a Bernoulli number
$B_{n}$ is represented by an umbra $B$ which should be replaced
by $B_{n}$ each time it appears as $B^{n}$.
\begin{prop}
\label{pro:[Gessel,-(7.2)]}{[}Gessel, (7.2){]} The Bernoulli numbers
satisfy, in the umbral notation,\[
\left(B+1\right)^{n}=B^{n},\,\,\forall n\ne1.\]
\end{prop}
\begin{proof}
In a probabilistic setting, this is a direct consequence of the result
of Proposition \ref{pro:proposition1} for $n\ne1,$ \[
B_{n}=E\left(\imath L-\frac{1}{2}\right)^{n}=E\left(\imath L+\frac{1}{2}\right)^{n}.\]
\end{proof}
\begin{prop}
\label{pro:The-Bernoulli-numbers}The Bernoulli numbers verify, in
the umbral notation,\[
\left(B+1\right)^{n}=\left(-1\right)^{n}B^{n},\,\, n\ne1.\]
\end{prop}
\begin{proof}
Since by symmetry \begin{eqnarray*}
E\left(\frac{\imath}{2\pi}\log\frac{E_{1}}{E_{2}}-\frac{1}{2}\right)^{n} & = & E\left(\frac{\imath}{2\pi}\log\frac{E_{2}}{E_{1}}-\frac{1}{2}\right)^{n}\\
 & = & \left(-1\right)^{n}E\left(\frac{\imath}{2\pi}\log\frac{E_{1}}{E_{2}}+\frac{1}{2}\right)^{n}=\left(-1\right)^{n}B_{n}\end{eqnarray*}
for $n\ge2,$ we deduce the result.
\end{proof}
Other integral representations of Bernoulli numbers and polynomials
exist in the litterature; for example, \eqref{eq:BnU} can be considered
as a substitute for the Mellin Barnes integral \cite[24.7.11]{NIST}
\[
B_{n}\left(x\right)=\frac{1}{2\pi\imath}\int_{-c-\imath\infty}^{-c+\imath\infty}\left(x+t\right)^{n}\left(\frac{\pi}{\sin\pi t}\right)^{2}dt;\]
moreover, the integral representation \cite[p.39, eq. (27)]{Bateman2}\[
B_{2n}=\left(-1\right)^{n+1}\pi\int_{0}^{+\infty}t^{2n}\text{csch}^{2}\left(\pi t\right)dt\]
can be easily deduced from \eqref{eq:Bn}; however, this identity
holds only for even order Bernoulli numbers - the odd order ones being
equal to $0$ except $B_{1}$, as a consequence of propositions \ref{pro:[Gessel,-(7.2)]}
and \ref{pro:The-Bernoulli-numbers}. This can be a difficulty in
the calculation of some series which involve all Bernoulli numbers.
As an example, we provide a short proof of Kaneko's theorem.
\begin{thm*}
{[}Kaneko{]} The Bernoulli numbers verify\[
\sum_{i=0}^{n+1}\binom{n+1}{i}\left(n+i+1\right)B_{n+i}=0,\,\, n\ge0.\]
\end{thm*}
\begin{proof}
It can be easily checked that $\forall z\in\mathbb{C},\,\, n\ge0,$\[
\sum_{i=0}^{n+1}\binom{n+1}{i}\left(n+i+1\right)z^{n+i}=\left(n+1\right)z^{n}\left(1+z\right)^{n}\left(1+2z\right)\]
so that \begin{eqnarray*}
\sum_{i=0}^{n+1}\binom{n+1}{i}\left(n+i+1\right)B_{n+i} & = & \left(n+1\right)E\left(\imath L-\frac{1}{2}\right)^{n}\left(\imath L+\frac{1}{2}\right)^{n}\left(2\imath L\right)\\
 & = & 2\imath\left(n+1\right)E\left[L\left(-L^{2}-\frac{1}{4}\right)^{n}\right]\end{eqnarray*}
The random variable $L$ being symmetric, i.e. $-L$ being distributed
as $+L$, the former expectation can be computed as\[
E\left[L\left(-L^{2}-\frac{1}{4}\right)^{n}\right]=E\left[-L\left(-L^{2}-\frac{1}{4}\right)^{n}\right]=0\]
what proves the result.
\end{proof}
Kaneko's theorem is a special case of Momiyama's identity \[
\left(-1\right)^{m}\sum_{k=0}^{m}\binom{m+1}{k}\left(n+k+1\right)B_{n+k}=\left(-1\right)^{n+1}\sum_{k=0}^{n}\binom{n+1}{k}\left(m+k+1\right)B_{m+k}\]
which can also be easily proved by remarking that\begin{eqnarray*}
\sum_{k=0}^{m}\binom{m+1}{k}\left(n+k+1\right)B_{n+k} & = & E_{L}\left(n+1\right)\left(\imath L-\frac{1}{2}\right)^{n}\left(\imath L+\frac{1}{2}\right)^{m+1}\\
 & + & \left(m+1\right)\left(\imath L-\frac{1}{2}\right)^{n+1}\left(\imath L+\frac{1}{2}\right)^{m}\\
 & - & \left(m+n+2\right)\left(\imath L-\frac{1}{2}\right)^{m+n+1}\end{eqnarray*}
which, by the symmetry of $L,$ can be shown to coincide with the
right-hand side sum.

Another famous identity can be proved easily using the probabilistic
approach
\begin{thm}
\label{thm:Bernoulli integration}The following identity holds\[
\int_{x}^{y}B_{n}\left(z\right)dz=\frac{B_{n+1}\left(y\right)-B_{n+1}\left(x\right)}{n+1},\,\, n\ge2\]
\end{thm}
\begin{proof}
Using the probabilistic representation \eqref{eq:Bn(x)Sun}, \begin{eqnarray*}
\int_{x}^{y}B_{n}\left(z\right)dz & = & E\int_{x}^{y}\left(\imath L+z-\frac{1}{2}\right)^{n}\\
 & = & \frac{1}{n+1}\left(E\left(\imath L+y+\frac{1}{2}\right)^{n+1}-E\left(\imath L+x+\frac{1}{2}\right)^{n+1}\right)\\
 & = & \frac{B_{n+1}\left(y\right)-B_{n+1}\left(x\right)}{n+1}.\end{eqnarray*}
 
\end{proof}
An equivalent version of Theorem \ref{thm:Bernoulli integration}
is \[
\frac{d}{dz}B_{n}\left(z\right)=nB_{n-1}\left(z\right)\]
which shows that Bernoulli polynomials are Appell polynomials.

A related result is the following \cite[24.13.2]{NIST}
\begin{thm}
The Bernoulli polynomials satisfy\begin{equation}
\int_{x}^{x+1}B_{n}\left(z\right)dz=x^{n}.\label{eq:Bnxx+1}\end{equation}
\end{thm}
\begin{proof}
From Thm \ref{thm:Bernoulli integration}, this integral is equal
to\[
\frac{B_{n+1}\left(x+1\right)-B_{n+1}\left(x\right)}{n+1}=\frac{E\left(iL+x+\frac{1}{2}\right)^{n+1}-E\left(iL+x-\frac{1}{2}\right)^{n+1}}{n+1}.\]
Expanding each power with the binomial formula yields\[
\frac{1}{n+1}\sum_{k=0}^{n+1}\binom{n+1}{k}x^{k}\left\{ E\left(iL+\frac{1}{2}\right)^{n+1-k}-E\left(iL-\frac{1}{2}\right)^{n+1-k}\right\} .\]
By \eqref{pro:[Gessel,-(7.2)]}, all terms vanish except for $k=n,$
which yields the result.
\end{proof}
Another identity, used by Gessel in the proof of Kaneko's formula,
and also proved in \cite{Chen} using the extended Zeilberger's algorithm,
reads
\begin{thm}
The Bernoulli numbers satisfy the following identity\[
\sum_{k=0}^{m}\binom{m}{k}B_{n+k}=\left(-1\right)^{m+n}\sum_{k=0}^{n}\binom{n}{k}B_{m+k.}\]
\end{thm}
\begin{proof}
The proof is straightforward using \eqref{eq:Bn(x)Sun} since the
left-hand side is \[
\sum_{k=0}^{m}\binom{m}{k}B_{n+k}=E\left(\imath L-\frac{1}{2}\right)^{n}\sum_{k=0}^{m}\binom{m}{k}\left(\imath L-\frac{1}{2}\right)^{k}=E\left(\imath L-\frac{1}{2}\right)^{n}\left(\imath L+\frac{1}{2}\right)^{m}\]
while the right-hand side is $ $$E\left(\imath L-\frac{1}{2}\right)^{m}\left(\imath L+\frac{1}{2}\right)^{n}.$
Since $L$ is distributed as $-L$,\begin{eqnarray*}
E\left(\imath L-\frac{1}{2}\right)^{n}\left(\imath L+\frac{1}{2}\right)^{m} & = & E\left(-\imath L-\frac{1}{2}\right)^{n}\left(-\imath L+\frac{1}{2}\right)^{m}\\
 & = & \left(-1\right)^{m+n}E\left(\imath L+\frac{1}{2}\right)^{n}\left(\imath L-\frac{1}{2}\right)^{m}\end{eqnarray*}
which concludes the proof.
\end{proof}
In \cite{Chen}, K.W. Chen considers the sequence of numbers $\left\{ K_{n}\right\} ,\,\, n\ge0$
defined as \[
K_{n}=\sum_{i=0}^{n}\binom{n}{i}B_{n+i+1},\]
and proves the following
\begin{thm}
{[}Chen{]} The sequence $\left\{ K_{n}\right\} $ satisfies the identities

\[
\sum_{k=0}^{n}\binom{2n-k}{k}\frac{2n}{2n-k}K_{k}=-B_{2n},\,\, n\ge1\]
and \[
\sum_{k=0}^{n-1}\binom{2n-k-1}{k}\frac{2n-1}{2n-k-1}K_{k}=B_{2n-1},\,\, n\ge1.\]
\end{thm}
\begin{proof}
We give a short proof of the first identity remarking that $K_{n}=E\left(\left(\imath L-\frac{1}{2}\right)^{n+1}\left(\imath L+\frac{1}{2}\right)^{n}\right)$
and that \[
\sum_{k=0}^{n}\binom{2n-k}{k}\frac{2n}{2n-k}X^{k}=\left(\frac{1}{2}+X-\frac{\sqrt{1-4X}}{2}\right)^{n}+\left(\frac{1}{2}+X+\frac{\sqrt{1-4X}}{2}\right)^{n}\]
so that \[
\sum_{k=0}^{n}\binom{2n-k}{k}\frac{2n}{2n-k}K_{k}=EU\left(\left(\frac{1}{2}+UV-\frac{\sqrt{1-4UV}}{2}\right)^{n}+\left(\frac{1}{2}+UV+\frac{\sqrt{1-4UV}}{2}\right)^{n}\right)\]
with $U=\imath L-\frac{1}{2}$ and $V=\imath L+\frac{1}{2}$ so that
$UV=-L^{2}-\frac{1}{4}$ and $\sqrt{1-4UV}=\imath L$. Finally,\begin{eqnarray*}
\sum_{k=0}^{n}\binom{2n-k}{k}\frac{2n}{2n-k}K_{k} & = & E\left(\imath L-\frac{1}{2}\right)\left(\left(-\imath L+\frac{1}{2}\right)^{2n}+\left(\imath L+\frac{1}{2}\right)^{2n}\right)\\
 & = & E\left(\imath L-\frac{1}{2}\right)^{2n+1}+E\left(\imath L-\frac{1}{2}\right)E\left(\imath L+\frac{1}{2}\right)^{2n}.\end{eqnarray*}
The first term is $B_{2n+1}=0$ while the second reads\[
E\left(\imath L-\frac{1}{2}\right)E\left(\imath L+\frac{1}{2}\right)^{2n}=\left(E\left(\imath L+\frac{1}{2}\right)-1\right)E\left(\imath L+\frac{1}{2}\right)^{2n}\]
which is equal to $B_{2n+1}-B_{2n}=-B_{2n}$ since $n\ge1.$ The proof
of the second identity is equally easy.
\end{proof}
A last and less easy identity is the following from \cite{Gessel2}.
\begin{thm}
{[}Gessel{]} Denote the power sum polynomial\begin{equation}
S_{k}\left(n\right)=\sum_{i=1}^{n}i^{k}=\sum_{i=0}^{k}\left(-1\right)^{k-i}\binom{k}{i}\frac{n^{i+1}}{i+1}B_{k-i}.\label{eq:Skn}\end{equation}
Then, with $a\in\mathbb{N},\,\, a\ge1$ and $n\in\mathbb{N},$\[
B_{n}=\frac{1}{a\left(1-a^{n}\right)}\sum_{k=0}^{n-1}a^{k}\binom{n}{k}B_{k}S_{n-k}\left(a-1\right).\]

\end{thm}
A quick proof can be given using the following
\begin{lem}
The power sum polynomial has integral representation $\forall k,\forall n\in\mathbb{N},$\[
S_{k}\left(n\right)=\int_{0}^{n}B_{k}\left(z+1\right)dz.\]
\end{lem}
\begin{proof}
This can be verified considering $n\in\mathbb{R}$ in the right-hand
side of \eqref{eq:Skn} so that \begin{eqnarray*}
\frac{d}{dn}S_{k}\left(n\right) & = & E_{L}\sum_{i=0}^{k}\binom{k}{i}n^{i}\left(-1\right)^{k-i}\left(\imath L-\frac{1}{2}\right)^{k-i}=E_{L}\left(-\imath L+n+\frac{1}{2}\right)^{k}\\
 & = & E_{L}\left(\imath L+n+\frac{1}{2}\right)^{k}\end{eqnarray*}
and, since $S_{k}\left(0\right)=0,$\[
S_{k}\left(n\right)=\int_{0}^{n}E_{L}\left(\imath L+z+\frac{1}{2}\right)^{n}dz=\int_{0}^{n}B_{n}\left(z+1\right)dz\]
and the result holds in particular $\forall n\in\mathbb{N}.$
\end{proof}
Another simple proof of this lemma can obviously be deduced from identity
\eqref{eq:Bnxx+1}.

Using this integral representation, we compute now\[
\sum_{k=0}^{n-1}a^{k}\binom{n}{k}B_{k}S_{n-k}\left(a-1\right)=\sum_{k=0}^{n}a^{k}\binom{n}{k}B_{k}S_{n-k}\left(a-1\right)-a^{n}\left(a-1\right)B_{n}.\]
The right-hand side sum is\[
E_{L_{1}}\sum_{k=0}^{n}a^{k}\binom{n}{k}\left(\imath L_{1}-\frac{1}{2}\right)^{k}\int_{0}^{a-1}B_{n-k}\left(z+1\right)dz=\int_{0}^{a-1}E_{L_{1},L_{2}}\left[a\left(\imath L_{1}-\frac{1}{2}\right)+\left(\imath L_{2}+z+\frac{1}{2}\right)\right]^{n}dz\]
which can be integrated as\[
\frac{1}{n+1}E_{L_{1},L_{2}}\left[\left(a\left(\imath L_{1}-\frac{1}{2}\right)-\left(\imath L_{2}+a-\frac{1}{2}\right)\right)^{n+1}-\left(a\left(\imath L_{1}-\frac{1}{2}\right)-\left(\imath L_{2}+\frac{1}{2}\right)\right)^{n+1}\right].\]
Expanding both $\left(n+1\right)$ powers, we obtain \[
\frac{1}{n+1}E_{L_{1},L_{2}}\sum_{k=0}^{n+1}\binom{n+1}{k}a^{n+1-k}\left(\left(\imath L_{1}+\frac{1}{2}\right)^{n+1-k}\left(\imath L_{2}-\frac{1}{2}\right)^{k}-\left(\imath L_{1}-\frac{1}{2}\right)^{n+1-k}\left(\imath L_{2}+\frac{1}{2}\right)^{k}\right).\]
But by \eqref{pro:[Gessel,-(7.2)]}, all terms cancel in this sum
except for $k=1$ and $k=n$ in which cases they add up, so that we
obtain\[
\frac{1}{n+1}\left(-\left(n+1\right)a^{n}B_{n}\right)+\frac{1}{n+1}\left(\left(n+1\right)aB_{n}\right)=B_{n}a\left(1-a^{n-1}\right)\]
which, adding the remaining term $-a^{n}\left(a-1\right)B_{n},$ yields
the result.

\section{Euler numbers and Euler polynomials}

\subsection{definition and characterisation}

We derive here analogous results for Euler numbers and Euler polynomials.
A generating function for the sequence $\left\{ E_{n}\right\} $ of
Euler numbers is\[
\sum_{n=0}^{+\infty}E_{n}\frac{t^{n}}{n!}=\text{sech}\left(t\right)\]
and for the Euler poynomials $\left\{ E_{n}\left(x\right)\right\} $\[
\sum_{n=0}^{+\infty}E_{n}\left(x\right)\frac{t^{n}}{n!}=\frac{2e^{tx}}{e^{t}+1}.\]
The Euler numbers are obtained as \[
E_{n}=\frac{1}{2^{n}}E_{n}\left(\frac{1}{2}\right).\]
First values are\[
E_{0}\left(x\right)=1;\,\, E_{1}\left(x\right)=x-\frac{1}{2};\,\, E_{2}\left(x\right)=x^{2}-x;\,\, E_{3}\left(x\right)=x^{3}-\frac{3}{2}x^{2}+\frac{1}{4}\]
and\[
E_{0}=1;\,\, E_{1}=0;\,\, E_{2}=-1;\,\, E_{3}=0;\,\, E_{4}=5.\]
 In \cite{Sun}, the following formula is derived
\begin{thm}
{[}Sun{]} If $L_{0}$ is defined as \begin{equation}
L_{0}=\sum_{k=1}^{+\infty}\frac{L_{k}}{\left(2k-1\right)\pi}\label{eq:L0-1}\end{equation}
where $L_{k}$ are independent and Laplace distributed, then the Euler
polynomials read\begin{equation}
E_{n}\left(x\right)=E\left(\imath L_{0}+x-\frac{1}{2}\right)^{n}\label{eq:EnL0}\end{equation}
and the Euler numbers\[
E_{n}=2^{n}E_{n}\left(\frac{1}{2}\right)=2^{n}E\left(\imath L_{0}\right)^{n}.\]

\end{thm}
We provide a more convenient characterization of the random variable
$L_{0}$ as follows.
\begin{thm}
The random variable $L_{0}$ follows the hyperbolic secant distribution\[
f_{L_{0}}\left(x\right)=\text{sech}\left(\pi x\right).\]
\end{thm}
\begin{proof}
The characteristic function of $L_{0}$ is\[
Ee^{\imath L_{0}t}=\text{sech}\left(\frac{t}{2}\right).\]
From \cite[1.9.1]{Bateman}, \[
\int_{0}^{+\infty}\text{sech}\left(ax\right)\cos\left(xt\right)dx=\frac{\pi}{2a}\text{sech}\left(\frac{\pi}{2a}t\right)\]
so that, with $a=\pi,$ the density of $L_{0}$ is\[
f_{L_{0}}\left(x\right)=\text{sech}\left(\pi x\right).\]
Thus $\pi L_{0}$ follows an hyperbolic secant distribution.
\end{proof}
We note from \cite{Devroye} that the random variable $L_{0}$ can
be obtained as \begin{equation}
L_{0}=\frac{1}{\pi}\log\vert C\vert=\frac{1}{\pi}\left(\log\vert N_{1}\vert-\log\vert N_{2}\vert\right)\label{eq:L0}\end{equation}
where $C$ is Cauchy distributed and $N_{1}$ and $N_{2}$ are two
independent standard Gaussian random variables.

The random variable $L_{0}$ is also a scale mixture of Gaussian with
mixing distribution given in \cite{Barndorff}, the moment generating
function of which reads\[
\varphi\left(s\right)=\frac{1}{\cos\left(\pi\sqrt{2s}\right)}.\]
At last, the random variable $\imath L_{0}$ has the same moments
as the Lévy area - that is the signed area - of a Brownian motion
$B_{t}$ for $0\le t\le1,$ see \cite{Levin}.

We leave to the reader the proofs of basic identities such as \begin{equation}
E_{n}\left(1-x\right)=\left(-1\right)^{n}E_{n}\left(x\right),\label{eq:Euler1-x}\end{equation}
\[
\sum_{r=0}^{n}\binom{n}{r}E_{r}\left(x\right)=E_{n}\left(x+1\right)\]
and \begin{equation}
\sum_{r=0}^{n}\binom{2n}{2r}E_{2r}=0\label{eq:Eulersum2r}\end{equation}
 as a consequence of the probabilistic representation \eqref{eq:EnL0}:
the first identity is a consequence of the symmetry of the distribution
of $L_{0},$ the second is obtained using the binomial theorem.

\subsection{links between Bernoulli and Euler polynomials}

\subsubsection{a summation identity}

An interesting identity that links Bernoulli and Euler polynomials
is the following \cite[p. xxxiii]{Gradshteyn}\[
B_{n}\left(x\right)=\frac{1}{2^{n}}\sum_{k=0}^{n}\binom{n}{k}B_{n-k}E_{k}\left(2x\right)\]
and its more general version \cite[24.14.5]{NIST} \[
B_{n}\left(\frac{x+y}{2}\right)=\frac{1}{2^{n}}\sum_{k=0}^{n}\binom{n}{k}B_{n-k}\left(x\right)E_{k}\left(y\right).\]
The proof reads\[
\frac{1}{2^{n}}\sum_{k=0}^{n}\binom{n}{k}B_{n-k}\left(x\right)E_{k}\left(y\right)=\frac{1}{2^{n}}E\left(\imath L_{e}+x-\frac{1}{2}+\imath L_{0}+y-\frac{1}{2}\right)^{n}=E\left(\imath\frac{L_{0}+L_{e}}{2}+\frac{x+y}{2}-\frac{1}{2}\right)^{n}\]
and to conclude, we need simply to show that $\frac{L_{0}+L_{e}}{2}\sim L_{e}$:
this can be deduced immediately from the identities \eqref{eq:L0-1}
and \eqref{eq:L}; or from\[
\frac{L_{0}+L_{e}}{2}=\frac{1}{2\pi}\log\vert C\vert+\frac{1}{2\pi}\log\sqrt{\frac{E_{1}}{E_{2}}}=\frac{1}{2\pi}\log\frac{\vert N_{1}\vert\sqrt{E_{1}}}{\sqrt{E_{2}}\vert N_{2}\vert}\]
where $C$ is Cauchy, $E$ is exponential and since $\frac{N}{\sqrt{2E}}$
is Laplace distributed and $\vert L\vert$ is exponentially distributed,
we deduce\[
\frac{L_{0}+L_{e}}{2}\sim\frac{1}{2\pi}\log\frac{E_{1}}{E_{2}}\sim L_{e}.\]

\subsubsection{an integral identity}

Another similar identity is \cite[24.13.3]{NIST}\begin{equation}
\int_{x}^{x+\frac{1}{2}}B_{n}\left(z\right)dz=\frac{E_{n}\left(2x\right)}{2^{n+1}}\label{eq:BnE2x}\end{equation}
which can be proved as follows\begin{eqnarray*}
\int_{x}^{x+\frac{1}{2}}B_{n}\left(z\right)dz & = & \frac{1}{n+1}\left\{ E\left(iL_{e}+x\right)^{n+1}-E\left(iL_{e}+x-\frac{1}{2}\right)^{n+1}\right\} \\
 & = & \frac{1}{n+1}\left\{ E\left(i\frac{L_{e}}{2}+i\frac{L_{0}}{2}+\frac{2x}{2}\right)^{n+1}-E\left(i\frac{L_{e}}{2}+i\frac{L_{0}}{2}+\frac{2x-1}{2}\right)^{n+1}\right\} \\
 & = & \frac{1}{2^{n+1}\left(n+1\right)}\left\{ E\left(\left(iL_{e}+\frac{1}{2}\right)+\left(iL_{0}+2x-\frac{1}{2}\right)\right)^{n+1}-E\left(\left(iL_{e}-\frac{1}{2}\right)+\left(iL_{0}+2x-\frac{1}{2}\right)\right)^{n+1}\right\} .\end{eqnarray*}
By property \eqref{pro:The-Bernoulli-numbers}, all terms cancel in
the binomial expansions of the $\left(n+1\right)$ powers, except
for $k=n$ and the only remaining term is \[
\frac{1}{2^{n+1}\left(n+1\right)}\binom{n+1}{n}E\left(iL_{0}+2x-\frac{1}{2}\right)^{n}=\frac{E_{n}\left(2x\right)}{2^{n+1}}.\]

\subsubsection{a consequence of \eqref{eq:BnE2x}}

As a consequence of the identity \eqref{eq:BnE2x}, we deduce the
identity\[
E_{n}\left(x+1\right)+E_{n}\left(x\right)=2x^{n}\]
as follows:\[
E_{n}\left(x+1\right)+E_{n}\left(x\right)=2^{n+1}\int_{\frac{x}{2}}^{\frac{x}{2}+1}B_{n}\left(z\right)dz=2^{n+1}\left(\frac{x}{2}\right)^{n}=2x^{n}\]
where we applied the result \eqref{eq:Bnxx+1}.

\section{Hermite polynomials}

In \cite{Gessel}, the study of Hermite polynomials involves the definition
of the umbra $M$ such that\[
e^{Mx}=e^{-x^{2}},\,\, x\in\mathbb{C}.\]

In the rest of this section, we provide a probabilistic interpretation
to this umbra, and using a master identity on Gaussian random variables,
derive new simple proofs of the results by Gessel.
\begin{thm}
The umbra $M$ is equivalent to the expectation\begin{equation}
f\left(Mx\right)=Ef\left(-\imath Zx\right)\label{eq:umbraMGaussian}\end{equation}
where $Z$ is a Gaussian random variable with zero mean and variance
$\sigma^{2}=2,$ which we denote as $Z\sim\mathcal{N}\left(0,2\right)$,
for any admissible function $f.$\end{thm}
\begin{proof}
It suffices to prove\[
Ee^{-iZx}=e^{-x^{2}}\]
which is nothing but the expression of the characteristic function
of the Gaussian random variable $Z.$
\end{proof}
We need the following master identity.
\begin{prop}
In umbral notation,\begin{equation}
e^{Mx+M^{2}y}=\frac{1}{\sqrt{1+4y^{2}}}\exp\left(-\frac{x^{2}}{1+4y^{2}}\right).\label{eq:masterHermite}\end{equation}
\end{prop}
\begin{proof}
Expressing\[
e^{Mx+M^{2}y}=E_{V}e^{Mx+MVy}=E_{V}e^{-\left(x+Vy\right)^{2}}\]
where $V\sim\mathcal{N}\left(0,2\right),$ we need to compute the
expectation\[
\frac{1}{2\sqrt{\pi}}\int_{-\infty}^{+\infty}e^{-\frac{V^{2}}{4}}e^{-\left(x+Vy\right)^{2}}dV.\]
We recognize here, up to a constant term $\frac{\sqrt{\pi}}{y},$
the convolution, evaluated at the point $\frac{x}{y}$, of a variable
$V\sim\mathcal{N}\left(0,2\right)$ with a random variable $W\sim\mathcal{N}\left(0,\frac{1}{2y^{2}}\right);$
it can easily be checked to be the Gaussian with variance $\mathcal{N}\left(0,1+4y^{2}\right)$,
i.e. \[
\frac{1}{\sqrt{1+4y^{2}}}\exp\left(-\frac{x^{2}}{1+4y^{2}}\right).\]

\end{proof}
As special cases, we recover
\begin{enumerate}
\item for $y=0,$ $e^{Mx}=\exp\left(-x^{2}\right)$
\item for $x=0,$ $e^{M^{2}x}=\frac{1}{\sqrt{1+4x^{2}}}$
\end{enumerate}
Here are now a few results from \cite{Gessel} that can be deduced
from the master identity \eqref{eq:masterHermite}.
\begin{enumerate}
\item the umbral notation for the Hermite polynomials\[
H_{n}\left(u\right)=\left(2u+M\right)^{n}\]
can be deduced from the identity \eqref{eq:umbraMGaussian} with $f\left(M\right)=\left(2u+M\right)^{n},$
so that \[
H_{n}\left(u\right)=E\left(2u+\imath N\right)^{n}\]
where $N\sim\mathcal{N}\left(0,2\right),$ which is a well-known result
(see for example \cite[p.49]{Janson})
\item the generating function \[
\sum_{n=0}^{+\infty}H_{2n}\left(u\right)\frac{x^{n}}{n!}=\frac{1}{\sqrt{1+4x}}\exp\left(-\frac{4u^{2}x}{1+4x}\right)\]
can be derived from \[
\sum_{n=0}^{+\infty}H_{2n}\left(u\right)\frac{x^{n}}{n!}=E_{N}\exp\left(4x\left(u+\imath N\right)^{2}\right)=E_{V}\exp\left(-\left(V\sqrt{x}-iu\sqrt{4x}\right)^{2}\right)\]
 and application of the master identity \eqref{eq:masterHermite}
\item the same approach yields the bivariate generating function\[
\sum_{m,n=0}^{+\infty}H_{2m+n}\left(u\right)\frac{x^{m}}{m!}\frac{y^{n}}{n!}=\frac{e^{-\frac{y^{2}}{4x}}}{\sqrt{1+4x}}\exp\left(\frac{\left(2u\sqrt{x}+\frac{y}{2\sqrt{x}}\right)^{2}}{1+4x}\right).\]

\end{enumerate}

\section{Carlitz and Zeilberger's Hermite polynomials}

In \cite{Gessel}, Gessel proposes an umbral study of Carlitz Hermite
polynomials\[
H_{m,n}\left(u,v\right)=\sum_{k=0}^{m\wedge n}\binom{m}{k}\binom{n}{k}k!u^{m-k}v^{n-k}\]
and of the Zeilberger's Hermite polynomials\[
H_{m,n}\left(w\right)=\sum_{k=0}^{m\wedge n}\binom{m}{k}\binom{n}{k}k!w^{k}.\]
with notation $m\wedge n=\min\left(m,n\right).$ In this aim, he defines
the umbr\ae\,\, $A$ and $B$ by\begin{equation}
A^{m}B^{n}=\delta_{m,n}m!,\,\,\forall m,n\in\mathbb{N}\label{eq:AmBn}\end{equation}
or equivalently as\begin{equation}
\exp\left(Ax+By\right)=\exp\left(xy\right)\,\,\,\forall x,y\in\mathbb{R}.\label{eq:exp(xy)}\end{equation}
The probabilistic counterpart of these umbr\ae \,\,is given by
the following result.
\begin{prop}
Let us consider a complex circular normal random variable $Z$ with
density\[
f_{Z}\left(z\right)=\frac{1}{\pi}\exp\left(-\vert z\vert^{2}\right).\]
The umbr\ae\,\, $A$ and $B$ in \eqref{eq:AmBn} and \eqref{eq:exp(xy)}
are identified with the expectations with respect to $Z$ and $\bar{Z}$
respectively.\end{prop}
\begin{proof}
It suffices to check that \[
EZ^{m}\bar{Z}^{n}=\delta_{m,n}m!\]
which is straightforward. 

Equivalently,\begin{eqnarray*}
E\exp\left(Zx+\bar{Z}y\right) & = & E\exp\left(\left(X+\imath Y\right)x+\left(X-\imath Y\right)y\right)\\
 & = & E\exp\left(X\left(x+y\right)\right)E\exp\left(\imath Y\left(x-y\right)\right)\\
 & = & \exp\frac{1}{4}\left(x+y\right)^{2}\exp-\frac{1}{4}\left(x-y\right)^{2}\end{eqnarray*}
so that the result holds.
\end{proof}
As a consequence, we deduce a probabilistic representation of Zeilberger's
Hermite polynomials as\[
H_{m,n}\left(u\right)=E\left(1+Z\right)^{m}\left(1+u\bar{Z}\right)^{n}\]
and of the Carlitz Hermite polynomials as\[
H_{m,n}\left(u,v\right)=u^{m-n}E\left(1+Z\right)^{m}\left(uv+\bar{Z}\right)^{n}.\]
 Using now classical results about complex Gaussian random variables,
we can recover easily some of the results derived by Gessel using
umbral calculus. We need the following master identity, the umbral
version of which is Lemma 5.3 in \cite{Gessel}.
\begin{lem}
If $Z$ is circular normal, then\begin{equation}
E\exp\left(Za+\bar{Z}b+Z^{2}u+Z\bar{Z}v+\bar{Z}^{2}w\right)=\frac{1}{\sqrt{\left(1-v\right)^{2}-4uw}}\exp\left(\frac{a^{2}w+b^{2}u+ab\left(1-v\right)}{\left(1-v\right)^{2}-4uw}\right).\label{eq:master}\end{equation}
\end{lem}
\begin{proof}
The proof is a consequence of the fact that the expectation in \eqref{eq:master}
can be rewritten as\[
E\exp\left(\tilde{Z}a+\tilde{\bar{Z}}b\right)\]
where $\tilde{Z}$ is a complex Gaussian variable with covariance
matrix\[
E\tilde{Z}\tilde{Z}^{t}=\Sigma=\frac{1}{\sqrt{\left(1-v\right)^{2}-4xu}}\left[\begin{array}{cc}
x & \frac{\left(1-v\right)}{2}\\
\frac{\left(1-v\right)}{2} & u\end{array}\right].\]
The latest expectation over $\tilde{Z}$ is nothing but the generating
function of $\tilde{Z}$ computed at the point $\left(a,b\right).$
\end{proof}
In order to illustrate the efficiency of the probabilistic approach,
we provide some quick derivations of the results obtained by Gessel.
\begin{enumerate}
\item the generating function of Zeilberger's Hermite polynomials reads\begin{eqnarray*}
\sum_{m,n=0}^{+\infty}H_{m,n}\left(w\right)\frac{x^{m}}{m!}\frac{y^{n}}{n!} & = & E_{Z}\exp\left(x\left(1+Z\right)\right)\exp\left(y\left(1+w\bar{Z}\right)\right)\\
 & = & \exp\left(x+y\right)E_{Z}\left(xZ+yw\bar{Z}\right)\\
 & = & \exp\left(x+y+wxy\right)\end{eqnarray*}
where the latest equality is a consequnce of \eqref{eq:exp(xy)}
\item the bilinear generating function of Zeilberger's Hermite polynomials
reads\begin{eqnarray*}
\sum_{m,n=0}^{+\infty}H_{m,n}\left(u\right)H_{m,n}\left(v\right)\frac{x^{m}}{m!}\frac{y^{n}}{n!} & = & E\exp x\left(1+Z_{1}\right)\left(1+Z_{2}\right)\exp\left(y\left(1+u\bar{Z}_{1}\right)\left(1+v\bar{Z}_{2}\right)\right)\\
 & = & E\exp\left(x\left(1+Z_{2}\right)+y\left(1+\bar{Z}_{2}v\right)\right)\\
 &  & \times\exp\left(xZ_{1}\left(1+Z_{2}\right)+uy\bar{Z}_{1}\left(1+v\bar{Z}_{2}\right)\right)\end{eqnarray*}
Taking first expectation on $Z_{1}$ and using formula \eqref{eq:exp(xy)},
then expectation on $Z_{2}$ and using formula \eqref{eq:master}
yields the result.
\item the following general formula holds\begin{eqnarray}
\sum_{i,j,k,l,m=0}^{+\infty}H_{i+2k+m,j+2l+m}\left(u\right)\frac{v^{i}}{i!}\frac{w^{j}}{j!}\frac{x^{k}}{k!}\frac{y^{l}}{l!}\frac{t^{m}}{m!} & = & \frac{1}{\sqrt{\left(1-ut\right)^{2}-4u^{2}xy}}\label{eq:Hi2kmj2lm}\\
\times\exp\left(\frac{\left(1+uw\right)^{2}x+\left(1+uv\right)^{2}y+4uxy+\left(1-ut\right)\left(v+w+t+uvw\right)}{\left(1-ut\right)^{2}-4u^{2}xy}\right)\nonumber \end{eqnarray}
The quintuple sum can be computed as\[
E\exp\left[v\left(1+Z\right)+w\left(1+u\bar{Z}\right)+x\left(1+Z\right)^{2}+y\left(1+u\bar{Z}\right)^{2}+t\left(1+Z\right)\left(1+u\bar{Z}\right)\right]\]
and using the master lemma \eqref{eq:master} and simple algebra,
we recover \eqref{eq:Hi2kmj2lm}.
\end{enumerate}

\section{Conclusion}

We have seen that all umbr\ae \,\,used in \cite{Gessel} have a
simple probabilistic expectation counterpart. Depending on the complexity
of the formula to prove, each approach has its advantages and drawbacks.
A future direction of work is to deepen the understanding of the link
between the umbral and the probabilistic approach.

\end{document}